\documentclass[letterpaper, 10 pt, conference]{ieeeconf}  
\usepackage{gen_settings}

\IEEEoverridecommandlockouts                              

\title{\LARGE \bf
Probabilistic Guarantees for Nonlinear Safety-Critical Optimal Control
}

\author{Prithvi Akella$^{*}$, Wyatt Ubellacker$^{*}$, and Aaron D. Ames$^{1}$
\thanks{This work was supported by the AFOSR Test and Evaluation Program, grant FA9550-19-1-0302}%
\thanks{*Both authors contributed equally.}
\thanks{$^{1}$All authors are with the California Institute of Technology
        {\tt\small \{pakella, wubellac, ames\}@caltech.edu}}%
}

\begin{document}

\maketitle
\thispagestyle{empty}
\pagestyle{empty}

\begin{abstract}

Leveraging recent developments in black-box risk-aware verification, we provide three algorithms that generate probabilistic guarantees on (1) optimality of solutions, (2) recursive feasibility, and (3) maximum controller runtimes for general nonlinear safety-critical finite-time optimal controllers.  These methods forego the usual (perhaps) restrictive assumptions required for typical theoretical guarantees, \textit{e.g.} terminal set calculation for recursive feasibility in Nonlinear Model Predictive Control, or convexification of optimal controllers to ensure optimality.  Furthermore, we show that these methods can directly be applied to hardware systems to generate controller guarantees on their respective systems.

\end{abstract}

\section{Introduction}
From Kalman till date, the pursuit of theoretical guarantees for optimal controllers has fascinated the controls and robotics communities alike~\cite{lewis2012optimal,kalman1960contributions,locatelli2002optimal,sethi2019optimal}.  This fascination arises as optimal controllers provide a natural way of expressing and segmenting disparate control objectives, as can be easily seen in works regarding model predictive control (MPC)~\cite{camacho2013model,rawlings2000tutorial,garcia1989model}, control barrier functions~\cite{ames2016control,xu2015robustness,grandia2021multi}, and optimal path planning~\cite{raja2012optimal,noreen2016optimal,riviere2020glas}, among others.  However, optimization problems becoming central to controller synthesis resulted in newer problems such as determining whether solutions exist, \textit{e.g.} recursive feasibility in MPC, determining the efficiency with which solutions can be identified to inform control loop rates, and determining the optimality of identified solutions in non-convex optimization settings.

Recent years have seen tremendous strides in answering these questions, but areas of improvement still exist.  For example, advances in Nonlinear MPC still require assumptions on the existence of control invariant terminal sets and stabilizing controllers for recursive feasibility, though identification of such items for general nonlinear systems remains a difficult problem~\cite{maiworm2015scenario,esterhuizen2020recursive,fang2022model,yu2015nonlinear,lucia2020stability}.  In general, determination of solution optimality for MPC problems is equivalent to solving the Hamilton-Jacobi-Bellman equation which is known to be difficult~\cite{kirk2004optimal}.  For path-planning problems, RRT* and other, sampling-based methods are known to be probabilistically complete, \textit{i.e.} they will produce the optimal solution given an infinite runtime, though sample-complexity results for sub-optimal solutions are few~\cite{elbanhawi2014sampling,noreen2016optimal,karaman2011anytime}.  Finally, there are similarly few theoretical results on the time complexity of these controllers on hardware systems, as such an analysis is heavily dependent on the specific hardware.

\newidea{Our Contribution:} Here, the authors believe recent results in black-box risk-aware verification might prove useful in generating theoretical statements on recursive feasibility, provable sub-optimality of results, and time complexity of the associated controllers on hardware systems, without the need for restrictive assumptions.  Our results are threefold.
\begin{itemize}
    \item We provide theoretical guarantees on the provable sub-optimality of percentile-based optimization procedures~\cite{akella2022sample} on producing input sequences for general, finite-time optimal control problems.
    \item We provide an algorithm for determining the probability with which a black-box controller is successively feasible on existing system hardware.
    \item We provide an algorithm to determine a probabilistic upper bound on hardware-specific controller runtimes.
\end{itemize}

\begin{figure}[t]
    \centering
    \includegraphics{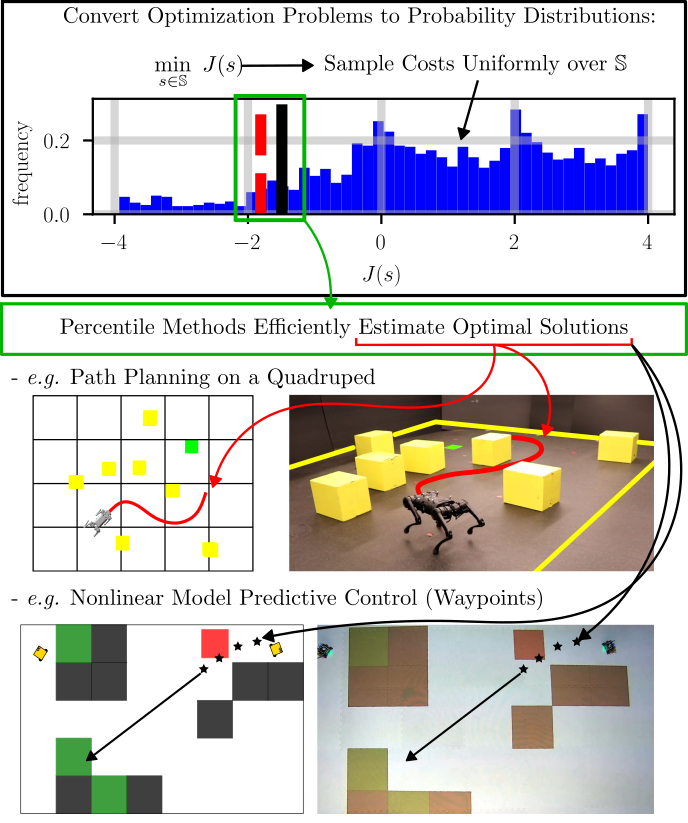}
    \caption{Finite-time optimal controllers and their guarantees can be expressed as optimization problems.  We provide probabilistic guarantees on solutions to these problems using novel results in black-box risk-aware verification.}
    \vspace{-0.2 in}
    \label{fig:title}
\end{figure}

\newidea{Structure:} To start, Section~\ref{sec:motivation} motivates and formally states the problems under study in this paper, and the introduction to Section~\ref{sec:all_guarantees} provides the general theorem employed throughout.  Then, Section~\ref{sec:percentile_input} details our algorithm that provides probabilistic guarantees on the optimality of outputted solutions to nonlinear safety-critical finite-time optimal control problems.  Likewise, Section~\ref{sec:recursive_feasibility} details our algorithm that provides probabilistic guarantees on successive feasibility for the same type of optimal controllers.  Finally, Section~\ref{sec:time_complexity} details our algorithm that provides probabilistic guarantees on maximum controller runtimes.  Lastly, we portray all our theoretical results on hardware, as described for the quadrupedal example in Section~\ref{sec:quadruped} and for the Robotarium in Section~\ref{sec:robotarium}~\cite{wilson2020robotarium}.

\section{General Motivation and Problem Statements}
\label{sec:motivation}
We assume the existence of a nonlinear discrete-time system whose dynamics $f$ are (potentially) unknown:
\begin{equation}
    \label{eq:general_sys}
    x_{k+1} = f(x_k,u_k,d),~x \in \mathcal{X},~u \in \mathcal{U},~d \in \mathcal{D}. 
\end{equation}
Here, $\mathcal{X}\subseteq \mathbb{R}^n$ is the state space, $\mathcal{U} \subseteq \mathbb{R}^m$ is the input space, and $\mathcal{D} \subseteq \mathbb{R}^p$ is the space of variable objects in our environment that we can control, \textit{e.g.} center locations of obstacles and goals for path-planning examples, variable wind-speeds for a drone, \textit{etc}.  Provided this dynamics information, a cost $J$, state constraints, and input constraints, one could construct a Nonlinear Model Predictive Controller of the following form (with $j \in [0,1,\dots,H-1]$):
\begin{align}
\label{eq:gen_NMPC} \tag{NMPC}
\mathbf{u}^* = & \argmin_{\mathbf{u}=(u^0,u^1,\dots,u^{H-1}) \in \mathcal{U}^H}~& & J(\mathbf{u}, x_k, d), \\
&~ \qquad \mathrm{subject~to~} & & x^{j+1}_{k} = f(x^j_k,u^j,d), \\
& & & x^0_k = x_k, \\
& & & x^{j+1}_k \in \mathcal{X}^{j+1}_k, \\
& & & u^j \in \mathcal{U}.
\end{align}
\noindent For the analysis to follow, however, we note that the general NMPC problem posed in~\eqref{eq:gen_NMPC} can be posed as the following Finite-Time Optimal Control Problem.
\begin{align}
\label{eq:general_FTOCP} \tag{FTOCP}
& \argmin_{\mathbf{u}=(u^0,u^1,\dots,u^{H-1}) \in \mathcal{U}^H}~& & J(\mathbf{u}, x_k, d), \\
& ~\qquad \mathrm{subject~to~} & & \mathbf{u} \in \mathbb{U}(x_k,d) \subseteq \mathcal{U}^H.
\end{align}
Here, $J$ is a bounded (perhaps) nonlinear cost function, and $\mathbb{U}$ is a set-valued function outputting a constraint space for input sequences that (potentially) depends on the initial system and environment states $(x_k,d)$, respectively.  Specific examples following this general form will be provided in Sections~\ref{sec:quadruped} and~\ref{sec:robotarium}.  Finally, $H > 0$ is the horizon length for the finite-time optimal control problem.  Then, the three problem statements predicated on this optimal controller~\eqref{eq:general_FTOCP} follow.  
\begin{problem}
\label{prob:percentile}
Develop a procedure to identify input sequences $\mathbf{u}$ that are in the $100(1-\epsilon)\%$-ile for some $\epsilon \in (0,1]$ with respect to solving~\eqref{eq:general_FTOCP}.
\end{problem}
\begin{problem}
\label{prob:recursive_feasibility}
Develop a procedure to determine whether~\eqref{eq:general_FTOCP} is recursively feasible.
\end{problem}
\begin{problem}
    \label{prob:runtimes}
    Develop a procedure to upper bound maximum controller runtimes for optimal controllers of the form in~\eqref{eq:general_FTOCP}.
\end{problem}

\section{Probabilistic Guarantees}
\label{sec:all_guarantees}
To make progress on the aforementioned problem statements --- each will be addressed in a separate subsection to follow --- we will first state a general result combining existing results on black-box risk-aware verification.  To that end, consider the following optimization problem:
\begin{equation}
    \label{eq:general_opt}
    \min_{\decisionvariable \in \decisionspace}~J(\decisionvariable),
\end{equation}
subject to the following assumption:
\begin{assumption}
\label{assump:percent_assump}
The decision space $\decisionspace$ is a set with bounded volume, \textit{i.e.} $\int_{\decisionspace}~1~ds = V_{\decisionspace} < \infty$ or $\decisionspace$ has a finite number of elements.  Furthermore, the cost function $J$ is bounded over $\decisionspace$, \textit{i.e.} $\exists~m,M \in \mathbb{R}, \suchthat m\leq J(\decisionvariable) \leq M,~\forall~\decisionvariable \in \decisionspace$.
\end{assumption}

This assumption permits us to define the functions $\mathcal{V}, F$ corresponding to the volume fraction occupied by a subset $A$ of $\decisionspace$ and the set of strictly better decisions for a provided decision $\decisionvariable' \in \decisionspace$, respectively:
\begin{gather}
    \label{eq:volume_fraction}
    \mathcal{V}(A) = \frac{\int_A~1~ds}{\int_{\decisionspace}~1~ds}, \\
    \label{eq:falsifying_set}
    F(\decisionvariable') = \{\decisionvariable \in \decisionspace~|~J(\decisionvariable) < J(\decisionvariable')\}.
\end{gather}
Naturally then, for a given decision $\decisionvariable' \in \decisionspace$, were $\mathcal{V}(F(\decisionvariable')) \leq \epsilon$ for some $\epsilon \in (0,1]$, \textit{i.e.} $\decisionvariable'$ is such that the volume fraction of strictly better decisions is no more than $\epsilon$, then $\decisionvariable'$ would be in the $100(1-\epsilon)\%$-ile with respect to minimizing $J$.  Likewise, the associated minimum cost of such a decision $J(\decisionvariable')$ should also be a probabilistic lower bound on achievable costs.  Both of these notions are expressed formally in the theorem below, which combines similar results from~\cite{akella2022sample,akella2022scenario}.
\begin{theorem}
\label{thm:prob_optimality}
Let $\{(\decisionvariable_i,J(\decisionvariable_i))\}_{i=1}^N$ be a set of $N$ decisions and costs for decisions $\decisionvariable_i$ sampled via $\uniform[\decisionspace]$, with $\zeta^*_N$ the minimum sampled cost and $\decisionvariable^*_N$ the (perhaps) non-unique decision with minimum cost.  Then $\forall~\epsilon \in [0,1]$, the probability of sampling a decision whose cost is at-least $\zeta^*_N$ is at minimum $1-\epsilon$ with confidence $1-(1-\epsilon)^N$, \textit{i.e.}
\begin{equation}
    \label{eq:prob_ver}
    \prob^N_{\uniform[\decisionspace]}
    \left[\prob_{\uniform[\decisionspace]}\left[J(\decisionvariable) \geq \zeta^*_N\right] \geq 1-\epsilon \right] \geq 1-(1-\epsilon)^N.
\end{equation}
Furthermore, $\forall~\epsilon \in (0,1]$, $\decisionvariable^*_N$ is in the $100(1-\epsilon)\%$-ile with minimum confidence $1-(1-\epsilon)^N$, \textit{i.e.}
\begin{equation}
    \label{eq:percent_opt}
    \prob^N_{\uniform[\decisionspace]}\left[\mathcal{V}(F(\decisionvariable^*_N)) \leq \epsilon \right] \geq 1-(1-\epsilon)^N.
\end{equation}
\end{theorem}
\begin{proof}
    This is a direct application of Theorem~7 in~\cite{akella2022sample} and Theorem~2 in~\cite{akella2022scenario}.
\end{proof}

To clarify then, this is the central result on probabilistic optimality --- derived from existing results on black-box risk-aware verification --- that we will exploit in the remainder of the paper to address the three aforementioned questions.  Our efforts regarding the first problem statement will follow.
\subsection{Percentile-Based Input Selection}
\label{sec:percentile_input}
Problem~\ref{prob:percentile} references the development of an efficient method to solve~\eqref{eq:general_FTOCP}.  To that end, we aim to take a percentile method that exploits equation~\eqref{eq:percent_opt} in Theorem~\ref{thm:prob_optimality}.  As a result, our corollary in this vein stems directly from Theorem~\ref{thm:prob_optimality}, though we will make one clarifying assumption.
\begin{assumption}
\label{assump:percent_solution}
Let $J$ and $\mathbb{U}$ be as per~\eqref{eq:general_FTOCP}, let $\mathcal{V}$ be as per~\eqref{eq:volume_fraction} with respect to the decision space $\mathbb{U}(x_k,d)$, and let $F$ be as per~\eqref{eq:falsifying_set} with respect to this cost $J$ and $\mathbb{U}(x_k,d)$.  Furthermore, let $J$ be bounded over $\mathbb{U}(x_k,d)$, and let $\mathbb{U}(x_k,d)$ be a set of bounded volume (or finitely many elements if a discrete set) for any choice of $(x_k,d) \in \mathcal{X}\times\mathcal{D}$ (these sets defined in~\eqref{eq:general_sys}).  Finally, let $\{(\mathbf{u}_i, J(\mathbf{u}_i,x_k,d))\}_{i=1}^N$ be a set of $N$ uniformly sampled sequences $\mathbf{u}_i$ from $\mathbb{U}(x_k,d)$ with their corresponding costs, and let $\mathbf{u}^*_N$ be the (potentially) non-unique sequence with minimum sampled cost.
\end{assumption}
\begin{corollary}
\label{corr:percent_opt_FTOCP}
Let Assumption~\ref{assump:percent_solution} hold and let $\epsilon \in (0,1]$.  Then, $\mathbf{u}^*_N$ is in the $100(1-\epsilon)\%$-ile with respect to minimizing $J$ at the current system and environment state $(x_k,d)$ with minimum confidence $1-(1-\epsilon)^N$, \textit{i.e.},
\begin{equation}
    \prob^N_{\uniform[\mathbb{U}(x_k,d)]}\left[\mathcal{V}(F(\mathbf{u}^*_N)) \leq \epsilon \right] \geq 1-(1-\epsilon)^N.
\end{equation}
\end{corollary}
\begin{proof}
Use equation~\eqref{eq:percent_opt} in Theorem~\ref{thm:prob_optimality}.
\end{proof}

\begin{figure}[t]
    \centering
    \includegraphics[width = \columnwidth]{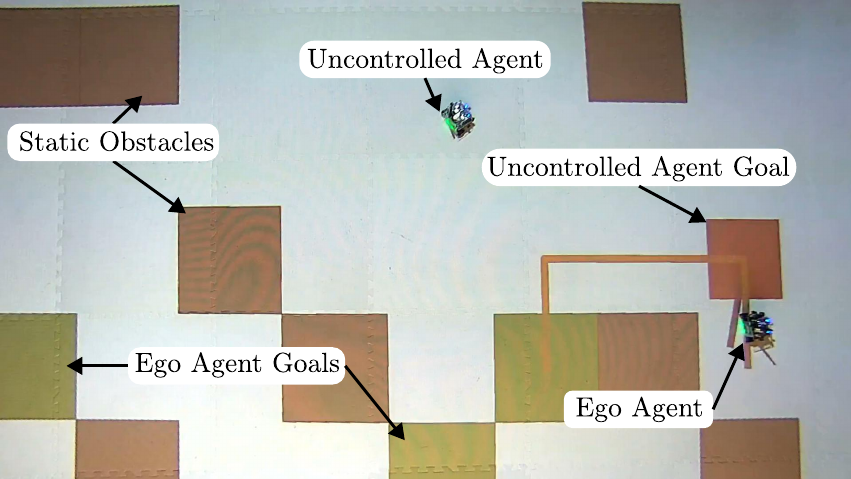}
    \caption{Experimental setup for Robotarium reach-avoid tests.}
    \vspace{-0.2 in}
    \label{fig:robotarium_setup}
\end{figure}

In short, Corollary~\ref{corr:percent_opt_FTOCP} tells us that if we have a finite-time optimal control problem of the form in~\eqref{eq:general_FTOCP}, where for some system and environment state $(x_k,d)$, the cost function $J$ is bounded over a bounded decision space $\mathbb{U}(x_k,d)$, then we can take a percentile approach to identify input sequences that are better than a large fraction of the space of all feasible input sequences.  Notably, this statement is made independent of the convexity, or lack thereof, of~\eqref{eq:general_FTOCP}, making it especially useful for non-convex MPC.  Furthermore, as is done in Section~\ref{sec:quadruped} to follow, one can further optimize over the outputted percentile solution $\mathbf{u}^*_N$ via gradient descent --- should gradient information be available.  The resulting solution then retains the same confidence on existing within the same percentile, while also being efficient to calculate.  This does introduce new questions, however.  Namely, will a percentile solution always exist, and how efficient is the calculation of these sequences on a given hardware?  These questions will be answered in the sections to follow.

\subsection{Determining Recursive Feasibility}
\label{sec:recursive_feasibility}
Problem~\ref{prob:recursive_feasibility} references the development of an algorithm to efficiently determine the recursive feasibility of~\eqref{eq:general_FTOCP}.  To ease the statement of the theoretical results to follow, we indicate via $|\mathbb{U}(x_k,d)|$ the ``size" of the constraint space $\mathbb{U}(x_k,d)$ for~\eqref{eq:general_FTOCP}, with $|\varnothing| = 0$.  Additionally, we will assume that there exists some controller $U$ that either utilizes the aforementioned percentile method in Section~\ref{sec:percentile_input} or some other technique to produce (potentially approximate) solutions to~\eqref{eq:general_FTOCP}, \textit{i.e.}
\begin{equation}
\label{eq:controller}
\exists~U: \mathcal{X} \times \mathcal{D} \to \mathcal{U} \suchthat U(x,d) = u \in \mathcal{U}    
\end{equation}
Furthermore, we will indicate via the following notation, the evolution of our system under this controller $U$, provided an initial system and environment state:
\begin{equation}
    x^+[x,d] = f(x,U(x,d),d).
\end{equation}
This allows us to formally define recursive feasibility.
\begin{definition}
\label{def:recursive_feasibility}
An optimal controller of the form in~\eqref{eq:general_FTOCP} is \textit{recursively feasible} if and only if for all system and environment states, the feasible space for~\eqref{eq:general_FTOCP} is non-empty for successive timesteps, \textit{i.e.} $\forall~(x,d) \in \mathcal{X} \times \mathcal{D},~|\mathbb{U}(x,d)| > 0 \implies |\mathbb{U}(x^+[x,d],d)| > 0$.
\end{definition}

\begin{figure}[t]
\begin{tikzpicture}[]
    \node[anchor=south west,inner sep=0] (image) at (0,0) {\includegraphics[width=\columnwidth]{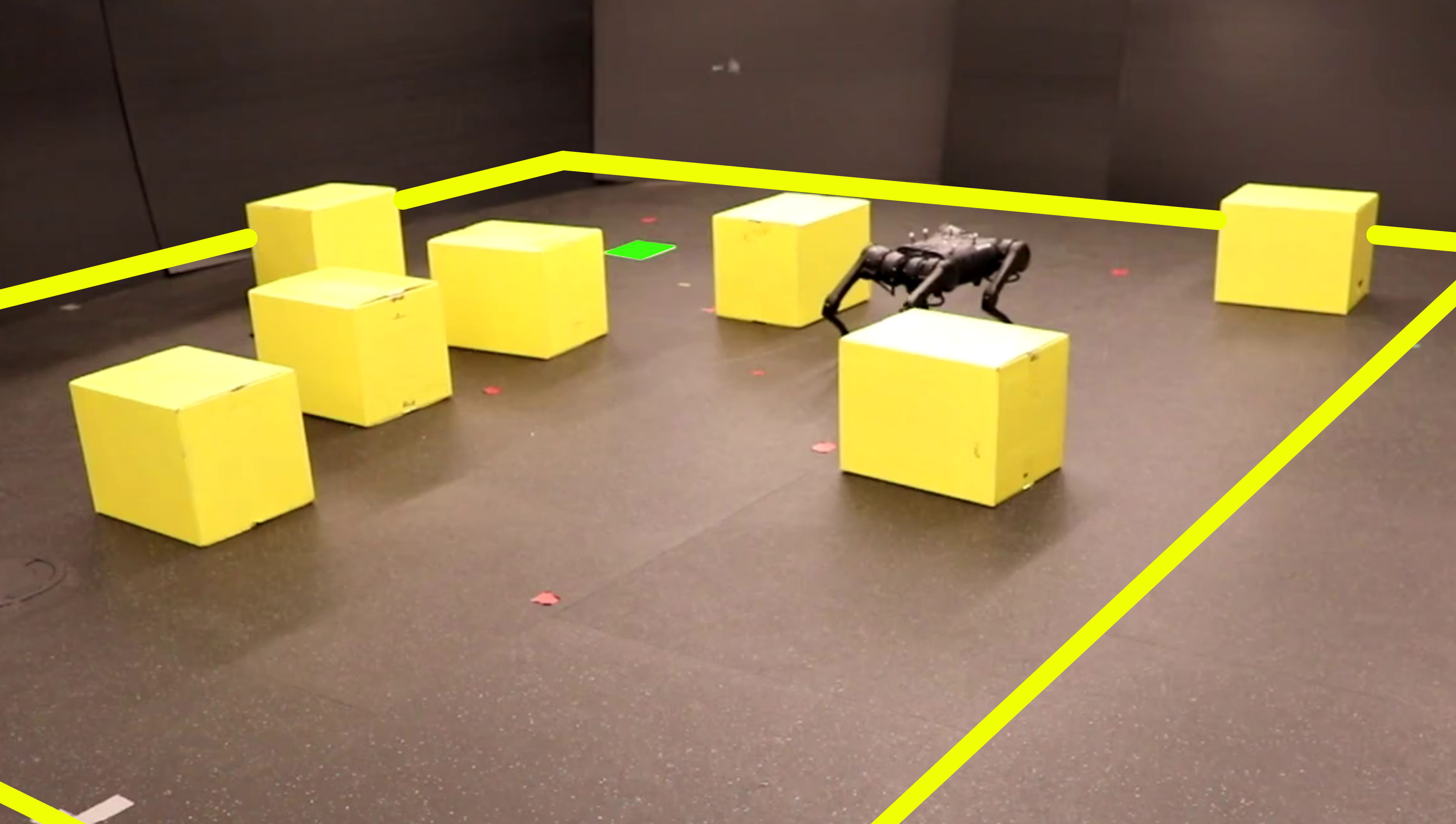}};
        \begin{scope}[x={(image.south east)},y={(image.north west)}]
        \node[] (quad) at (0.66,0.7) {};
        \node[fill=white,rounded corners=1mm] (quad_lbl) at (0.8,0.9) {Agent};
        \node[] (goal) at (0.44,0.72) {};
        \node[fill=white,rounded corners=1mm] (goal_lbl) at (0.3,0.9) {Goal};
        \node[] (obs_1) at (0.2,0.35) {};
        \node[] (obs_2) at (0.6,0.4) {};
        \node[fill=white,rounded corners=1mm] (obs_lbl) at (0.4,0.2) {Obstacles};
        \draw [white,-latex,line width=1mm,line cap=round] (quad_lbl) -- (quad);
        \draw [white,-latex,line width=1mm,line cap=round] (goal_lbl) -- (goal);
        \draw [white,-latex,line width=1mm,line cap=round] (obs_lbl) -- (obs_1);
        \draw [white,-latex,line width=1mm,line cap=round] (obs_lbl) -- (obs_2);
    \end{scope}
\end{tikzpicture}
    \caption{Experimental setup for Quadruped reach-avoid tests.}
    \vspace{-0.2 in}
    \label{fig:quad_setup}
\end{figure}
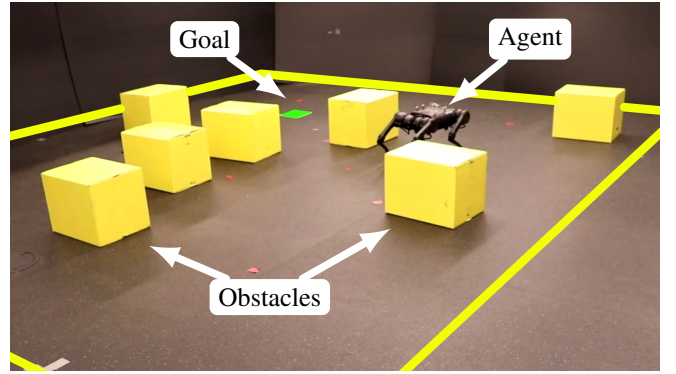

As motivated earlier, we can express recursive feasibility determination as an optimization problem.  Specifically, let our cost function $C$ be as follows:
\begin{gather}
    \mathbb{T}(x,d) = |\mathbb{U}(x,d)| > 0~\mathrm{and}~|\mathbb{U}(x^+[x,d],d)|>0, \\
    \label{eq:recursive_feasibility_cost}
    C(x,d) =
    \begin{cases}
    1 & \mbox{if~} \mathbb{T}(x,d) = \mathrm{True},
    \\
    0 & \mbox{else}.
    \end{cases}
\end{gather}
We can generate a minimization problem provided this cost function $C$ over the joint state space $\mathcal{X} \times \mathcal{D}$:
\begin{equation}
    \label{eq:opt_recursive_feasibility}
    \min_{x \in \mathcal{X},~d \in \mathcal{D}}~C(x,d).
\end{equation}
If the solution to~\eqref{eq:opt_recursive_feasibility} were positive, then~\eqref{eq:general_FTOCP} is recursively feasible.  Likewise, if the solution were negative, then there exists a counterexample.  As a result, not only can we express recursive feasibility determination as an optimization problem, but this problem is also of the same form as in~\eqref{eq:general_opt}, permitting a probabilistic solution approach as expressed in the following assumption and corollary.
\begin{assumption}
    \label{assump:recursive_feasible}
    Let $C$ be as per~\eqref{eq:recursive_feasibility_cost}, let $\mathcal{X},\mathcal{D}$ be as per~\eqref{eq:general_sys} and also be spaces of bounded volume, let $\{C(x_i,d_i)\}_{i=1}^N$ be a set of $N$ cost evaluations of decision tuples $(x_i,d_i)$ sampled independently via $\uniform[\mathcal{X} \times \mathcal{D}] \triangleq \mu$, let $\zeta^*_N$ be the minimum cost evaluation, and let $\epsilon \in [0,1]$.
\end{assumption}
\begin{corollary}
    \label{corr:recursive_feasible}
    Let Assumption~\ref{assump:recursive_feasible} hold.  Then if $\zeta^*_N = 1$,~\eqref{eq:general_FTOCP} is successively feasible with minimum probability $1-\epsilon$ and with minimum confidence $1-(1-\epsilon)^N$.
\end{corollary}
\begin{proof}
    Equation~\eqref{eq:prob_ver} in Theorem~\ref{thm:prob_optimality} tells us that
    \begin{equation}
    \prob^N_{\mu}\left[\prob_{\mu}\left[C(x,d) \geq \zeta^*_N\right] \geq 1-\epsilon \right] \geq 1-(1-\epsilon)^N.
    \end{equation}
    By definition of $C$ in~\eqref{eq:recursive_feasibility_cost}, if $\zeta^*_N = 1$, then with minimum probability $1-\epsilon$, $|\mathbb{U}(x,d)| > 0\implies |\mathbb{U}(x^+[x,d],d)|>0$.  In other words, with minimum probability $1-\epsilon$, if~\eqref{eq:general_FTOCP} were feasible at the prior time step, then it will also be feasible at the next time step, \textit{i.e.} successively feasible.
\end{proof} 

In other words, Corollary~\ref{corr:recursive_feasible} tells us that to probabilistically determine whether a given finite-time optimal control problem is successively feasible, it is sufficient to identify at least one input in the constraint space for successive optimization problems starting at $N$ randomly sampled state pairs $(x,d)$.  Determining at least one such input could be achieved by querying the corresponding controller $U$ or some other method desired by the practitioner.  Notably, this does not guarantee recursive feasibility as that would correspond to the optimal value of~\eqref{eq:opt_recursive_feasibility} being positive.  However, with arbitrarily high probability, we can provide guarantees that even hardware controllers will be successively feasible for sampled state pairs $(x,d) \in \mathcal{X} \times \mathcal{D}$, which is the underlying requirement for recursive feasibility as per Definition~\ref{def:recursive_feasibility}.

\subsection{Determining Hardware-Specific Controller Runtimes}
\label{sec:time_complexity}

Lastly, Problem~\ref{prob:runtimes} references the development of an algorithm to efficiently identify maximum controller runtimes on existing system hardware.  To address this from a probabilistic perspective, we will first define some notation.  To start, we will use the same controller $U$ as per equation~\eqref{eq:controller}.  We also denote via $T$ a timing function that outputs the evaluation time for querying the controller $U$ at a given state pair $(x,d)$, \textit{i.e.} $T:\mathcal{X} \times \mathcal{D} \to \mathbb{R}_{++}$.  Then we can nominally express maximum controller runtime determination as an optimization problem:
\begin{equation}
        \label{eq:timing_opt}
        \max_{x \in \mathcal{X},~d \in \mathcal{D}}~T(x,d).
\end{equation}
\noindent Under the fairness assumption that the controller does have a bounded runtime, however, identification of a probabilistic maximum runtime is solvable via probabilistic optimization procedures as outlined by Theorem~\ref{thm:prob_optimality}.  In a similar fashion as prior, we will state a clarifying assumption and the formal corollary statement will follow.

\begin{assumption}
    \label{assump:runtimes}
    Let $T$ be as per~\eqref{eq:timing_opt}, let $\mathcal{X},\mathcal{D}$ be as per~\eqref{eq:general_sys} and be of bounded volume, let $\{T(x_i,d_i)\}_{i=1}^N$ be a set of $N$ controller runtimes for state pairs $(x_i,d_i)$ sampled independently via $\uniform[\mathcal{X} \times \mathcal{D}] \triangleq \mu$, let $\zeta^*_N$ be the maximum runtime, and let $\epsilon \in [0,1]$.
\end{assumption}

\begin{corollary}
    \label{corr:run_times}
    Let Assumption~\ref{assump:runtimes} hold.  Then, the probability of sampling a state pair whose controller runtime is at most $\zeta^*_N$ is at-least $1-\epsilon$ with confidence $1-(1-\epsilon)^N$, \textit{i.e.}
    \begin{equation}
        \prob^N_{\mu}\left[\prob_{\mu}\left[T(x,d) \leq \zeta^*_N\right] \geq 1-\epsilon \right] \geq 1-(1-\epsilon)^N.
    \end{equation}
\end{corollary}
\begin{proof}
    Consider~\eqref{eq:timing_opt} expressed as a minimization.  Under the same assumptions, equation~\eqref{eq:prob_ver} in Theorem~\ref{thm:prob_optimality} states that
    \begin{equation}
        \prob^N_{\mu}\left[\prob_{\mu}\left[-T(x,d) \geq -\zeta^*_N\right] \geq 1-\epsilon \right] \geq 1-(1-\epsilon)^N,
    \end{equation}   
    and flipping the innermost inequality provides the result.
\end{proof}

\noindent In short then, Corollary~\ref{corr:run_times} tells us that probabilistic determination of maximum controller runtimes stems easily by recording controller runtimes for $N$ randomly sampled scenarios identified through $N$ randomly sampled system and environment state pairs $(x,d)$ from $\mathcal{X} \times \mathcal{D}$.

\section{Experimental Demonstrations}
\label{sec:experiments}
To demonstrate the contributions of our work, we applied the aforementioned methods to two reach-avoid navigation examples: 1) an A1 Unitree Quadruped \cite{Unitree} in a field of static obstacles, and 2) a Robotarium \cite{wilson2020robotarium} scenario with the controlled agent subject to both static obstacles and an additional uncontrolled, dynamic agent. 

\begin{figure*}[t]
\centering
\begin{subfigure}[t]{0.28\textwidth}
  \centering
  \includegraphics[width=0.95\columnwidth]{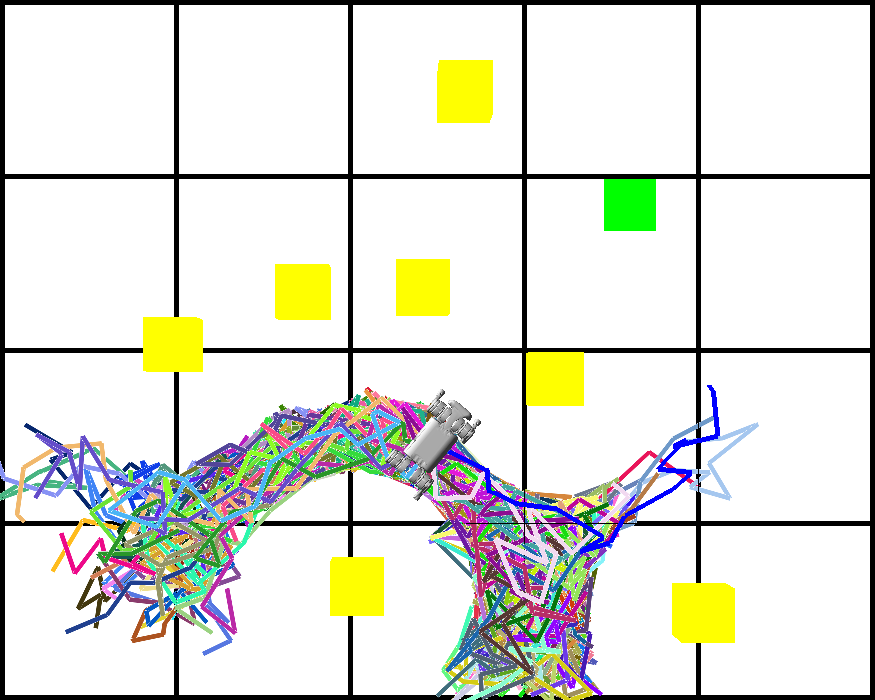}
  \caption{$\{(\mathbf{u}_i,J(\mathbf{u}_i,x_k,d))\}_{i=1}^N$}
\end{subfigure}%
\begin{subfigure}[t]{0.08\textwidth}
\centering
\begin{tikzpicture}
\node[anchor=center] at (0,1.65) {\Huge$\rightarrow$};
\node[anchor=center] at (0,0) {};
\end{tikzpicture}
\end{subfigure}%
\begin{subfigure}[t]{0.28\textwidth}
  \centering
  \includegraphics[width=0.95\columnwidth]{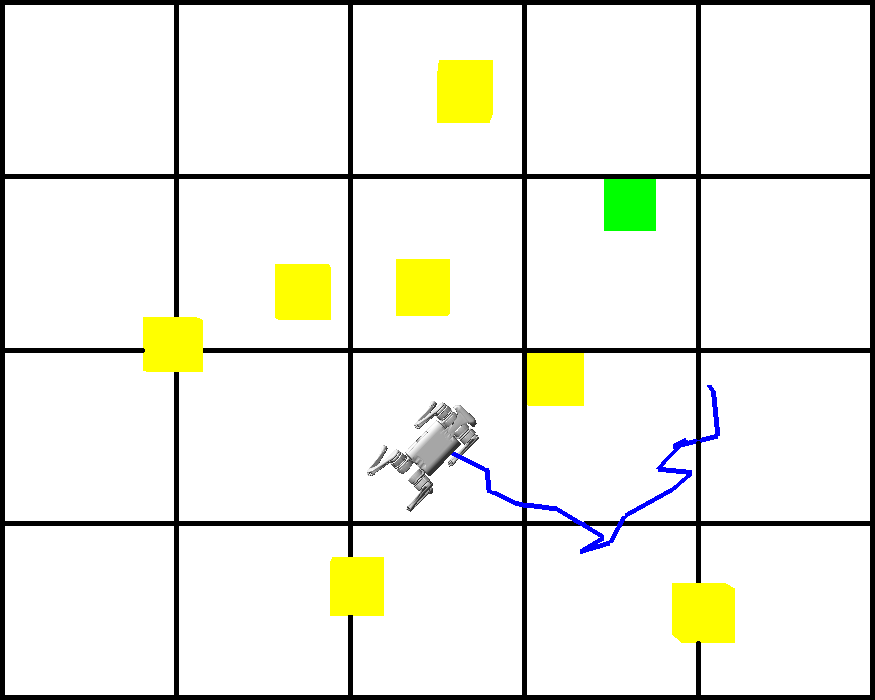}
  \caption{$\argmin\limits_{\mathbf{u}\in\{(\mathbf{u}_i,J(\mathbf{u}_i,x_k,d_k))\}_{i=1}^N} J(\mathbf{u}, x_k, d_k)$}
\end{subfigure}%
\begin{subfigure}[t]{0.08\textwidth}
\centering
\begin{tikzpicture}
\node[anchor=center] at (0,1.65) {\Huge$\rightarrow$};
\node[anchor=center] at (0,0) {};
\end{tikzpicture}
\end{subfigure}%
\begin{subfigure}[t]{0.28\textwidth}
  \centering
  \includegraphics[width=0.95\columnwidth]{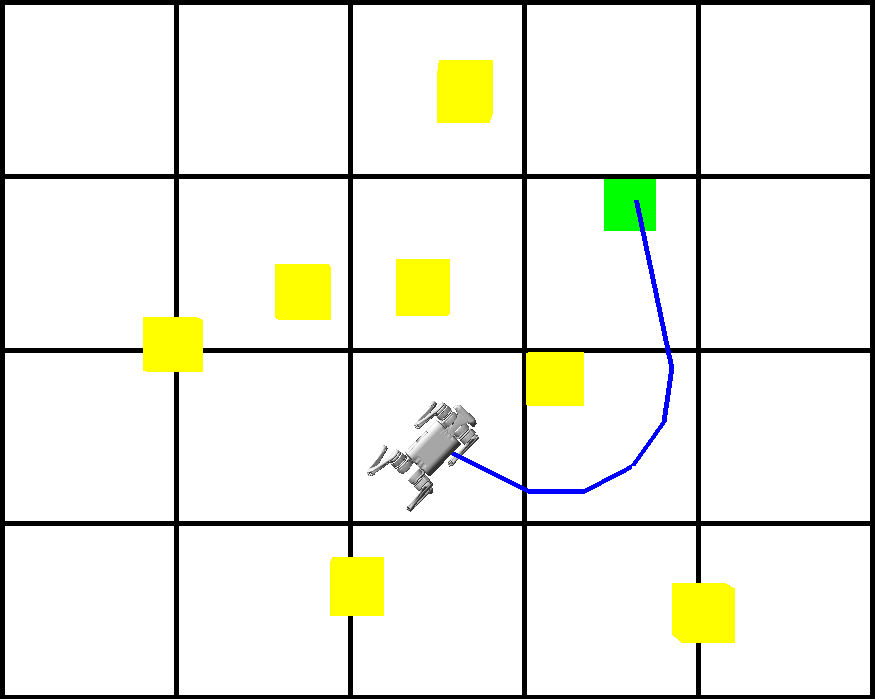}
  \caption{$\mathbf{u}_{n+1} = \mathbf{u}_n - \nabla J(\mathbf{u}_n)$ such that $\mathbf{u} \in \mathbb{U}(x_k,d_k) \subseteq \mathcal{U}^H$}
\end{subfigure}
\caption{Solving the FT-OCP for the quadruped reach-avoid experiment. (a) generates uniformly random feasible input sequence samples. (b) selects the best sample according to cost function $J(\mathbf{u}, x_k, d_k)$. Finally, (c) leverages the differentiability of $J$ to further improve the choice of $\mathbf{u}$ via constrained gradient descent.}
\label{fig:quad_ft-ocp_solve}
\end{figure*}
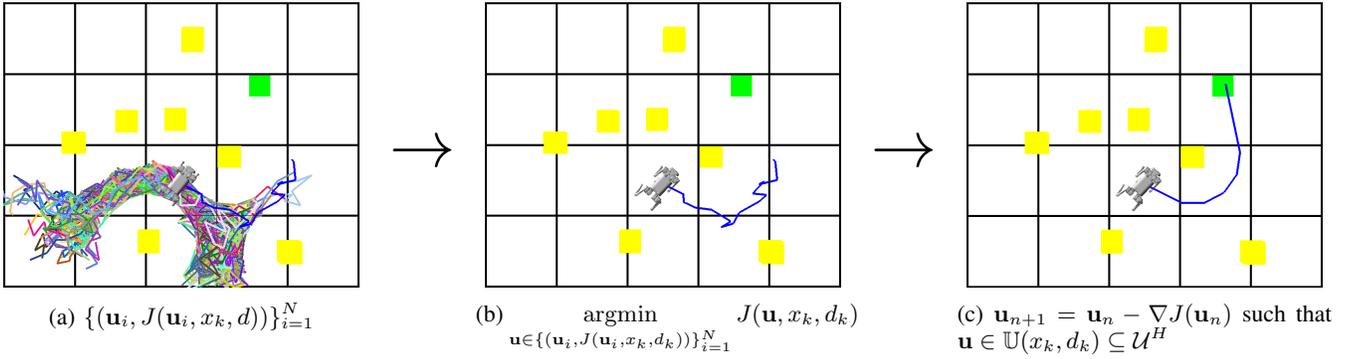

\subsection{Quadrupedal Walking}
\label{sec:quadruped}

\newidea{Reach Avoid Navigation Task:}
In the quadruped example, the agent is tasked to reach a specific goal location (green) while avoiding static obstacles (yellow) within a 5m by 4m space--the agent and obstacles move and can be placed continuously within this space.  The set of all environments $\mathcal{D}$ corresponds to the set of all setups, including goals, robot starting locations, and obstacles, that satisfy the aforementioned conditions while allowing for at-least one feasible path to the goal. Figure~\ref{fig:quad_setup} depicts an example setup, with Figure~\ref{fig:example_quad_setup} showing multiple examples of viable environments in $\mathcal{D}$.

\newidea{FT-OCP formulation:}
We formulated quadrupedal navigation as an optimal control problem of the form in~\eqref{eq:general_FTOCP}.  We consider as states, the position of the robot within a bounded rectangle $\mathcal{X} = [0,5] \times [0,4]$. Individual inputs are discrete changes in position with bounded magnitude, with corresponding $H$-length input sequence $\mathbf{u}$ a finite horizon of positional waypoints.  Mathematically, the state-dependent subset of permissible sequences $\mathcal{U}^H_p(x)$ is as follows, with $j\in[0,1,\dots,H-2]$:
\begin{equation}
    \mathcal{U}^H_p(x) = \left\{\mathbf{u} \in \mathcal{U}^H~\Bigg|
    \begin{gathered}
    \|u^0 - x\| \leq 0.03,~\mathrm{and~},\\
    \|u^{j+1}-u^j\| \leq 0.03.
    \end{gathered}
    \right\}
\end{equation}
$\mathbb{U}(x_k,d)$ then further constrains $\mathbf{u}$ to remain within a feasible set of states via a discrete barrier-like condition. To define that feasible state set, for $D$ obstacle positions let $d =[d_1^T,d_2^T,\dots,d_D^T]^T \in \mathbb{R}^{2\times D}$.  Then with a collision radius $r$, the feasible state set is:
\begin{align}
\mathcal{F}(d) = \{x\in \mathcal{X} \;|\; ||x-d_j|| \geq r\; \forall j=1,...,D\}.
\end{align}
Then we can define the overall constrained input space $\mathbb{U}(x,d)$ as follows, with $x^0 = x$, $x^{j+1} = f(x^j,u^j,d)$, and $\forall~\ell \in 0,1,\dots,H$:
\begin{equation}  
\label{eq:quad_feasible_set}
\mathbb{U}(x,d) = \left\{\mathbf{u}\in \mathcal{U}^H_p(x)~|~ x^{\ell} \in \mathcal{F}(d)\right\}.
\end{equation}
Here, the discrete-time dynamics are simply $f(x,u,d) = x + u$. Finally, with goal state $x_d$, we have our cost function $J$ as follows, again with $x^0 = x$ and $x^{j+1} = f(x^j,u^j,d)$:
\begin{align}
\label{eq:quad_cost}
J(\mathbf{u},x,d) = 10||x^H - x_d|| + \sum_{i=0}^{H-1} ||x^{i+1}-x^i||.
\end{align}
This cost simultaneously rewards the final waypoint when closer to the goal and a shorter overall path length.  As a result, the overall finite-time optimal control problem is:
\begin{align}
\label{eq:quad_FTOCP} 
\mathbf{u}^* = & ~\argmin_{\mathbf{u} \in \mathcal{U}^H}~& & J(\mathbf{u}, x_k, d)~\mathrm{as~per~}\eqref{eq:quad_cost}, \\
& \mathrm{subject~to~} & & \mathbf{u} \in \mathbb{U}(x_k,d)~\mathrm{as~per~}\eqref{eq:quad_feasible_set}.
\end{align}

\newidea{Solving the FT-OCP:}
To solve~\eqref{eq:quad_FTOCP}, we employ the procedure described in Section \ref{sec:percentile_input}.  We directly sample the input space $\mathcal{U}^H$ and employ rejection sampling to generate samples $\mathbf{u} \in \mathbb{U}(x_k,d)$, until we collect $1000$ such samples. From this collection of samples, we choose the minimum cost sample by evaluating $J(\mathbf{u},x_k,d)$. This sample meets our guarantees as described in Corollary~\ref{corr:percent_opt_FTOCP}. However, we recognize that our cost function is differentiable in $\mathbf{u}$, and we can employ constrained gradient descent \cite{boyd2003subgradient} to further improve the solution. This process is illustrated in Figure~\ref{fig:quad_ft-ocp_solve}.

\newidea{Experiments and Results:} 
Tests were performed for both random and curated obstacle locations, with care taken to reject samples without a feasible path to the goal. The quadruped was given a random start position and orientation, and a fixed goal, $x_d$. \eqref{eq:quad_FTOCP} was solved using a Python implementation of the above procedure at $\sim$1.5 Hz, taking $x_k$ to be the position of the quadruped as measured by an Optitrack motion capture system. An IDQP-based walking controller \cite{ubellacker2023icra} tracked the computed plan, with tangent angles along the plan used as desired quadruped heading.

By Corollary~\ref{corr:percent_opt_FTOCP}, choosing the best out of $1000$ uniformly chosen waypoint sequences implies that the best sequence $\mathbf{u}^*_N$ should be in the $99\%$-ile with $99.995\%$ confidence.  This is indeed the case as can be seen in the data portrayed at the top of Figure~\ref{fig:quadruped_data}, corroborating Corollary~\ref{corr:percent_opt_FTOCP}.  Both Corollaries~\ref{corr:recursive_feasible} and~\ref{corr:run_times} were also corroborated by recording successive feasibility and controller runtimes for $1000$ randomized instances of the percentile method applied to~\eqref{eq:quad_FTOCP}.  In all cases, the controller was successively feasible, and the maximum controller runtime was $0.92$ seconds.  Comparing against another $5000$ random samples affirms that the reported maximum runtime exceeded the $99\%$-ile cutoff, while the controller was successively feasible in all instances as well.  The data for runtimes is shown on the bottom in Figure~\ref{fig:quadruped_data}.  Qualitatively speaking, however, the proposed procedure produces a valid, collision-free plan in all tested scenarios. This plan ultimately leads to the quadruped reaching the desired goal in many scenarios. However, some obstacle placements lead to local minima that cannot be escaped, as this is a finite-time method. Increasing the horizon $H$ allows for success in these conditions, but requires a trade-off in execution time. These results are elucidated in the supplemental video. 

\subsection{Multi-Agent Verification}
\label{sec:robotarium}
Figure~\ref{fig:robotarium_setup} depicts the reach-avoid scenario for the Robotarium~\cite{wilson2020robotarium} agents which can be modeled as unicycle systems, \textit{i.e.} with $x_k \in \mathcal{X},~u_{k} \in \mathcal{U}$:
\begin{equation}
    \label{eq:unicyle_model}
    \begin{aligned}
        x_{k+1} & = \underbrace{x_k + (\Delta t = 0.033)
        \begin{bmatrix}
        \cos\left(x_k[3]\right) & 0 \\
        \sin\left(x_k[3]\right) & 0 \\
        0 & 1
        \end{bmatrix}u_k}_{f(x_k,u_k,d)}.
    \end{aligned}
\end{equation}
Here, $\mathcal{X} = [-1.6,1.6] \times [-1.2,1.2] \times [0,2\pi]$ and $\mathcal{U} = [-0.2,0.2] \times [\frac{-\pi}{2},\frac{\pi}{2}]$.  Additionally, each agent comes equipped with a Lyapunov controller $U$ that steers the agent to a provided waypoint $w \in \mathcal{W}$:
\begin{gather}
    U: \mathcal{X} \times \mathcal{D} \times \mathcal{W} \triangleq [-1.6,1.6] \times [-1.2,1.2] \to \mathcal{U}.
\end{gather}
The environment space $\mathcal{D}$ consists of the grid locations of $8$ static obstacles on an $8 \times 5$ grid overlaid on the state space $\mathcal{X}$, the cells of $3$ goals on the same grid, the starting position in $\mathcal{X}$ of another, un-controlled moving agent that is at-least $0.3$ meters away from the ego agent of interest, and the un-controlled agent's goal cell on the same grid.  No static obstacles are allowed to overlap with any of the goals, though the un-controlled agent's goal may overlap with at least one of the goals of the ego agent, and the setup of static obstacles must always allow for there to exist at least one path to one of the ego agent's goals.  Figure~\ref{fig:example_sim_setup} shows multiple examples of environment setups within $\mathcal{D}$.

\begin{figure}[t]
    \centering
    \includegraphics[width = \columnwidth]{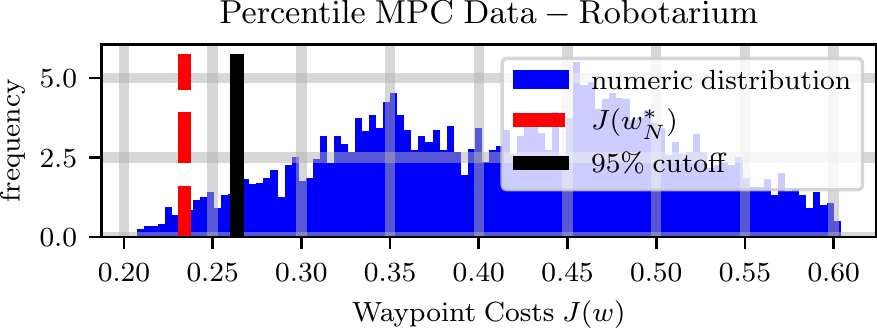} \vspace{0.025 in}\\
    \includegraphics[width = \columnwidth]{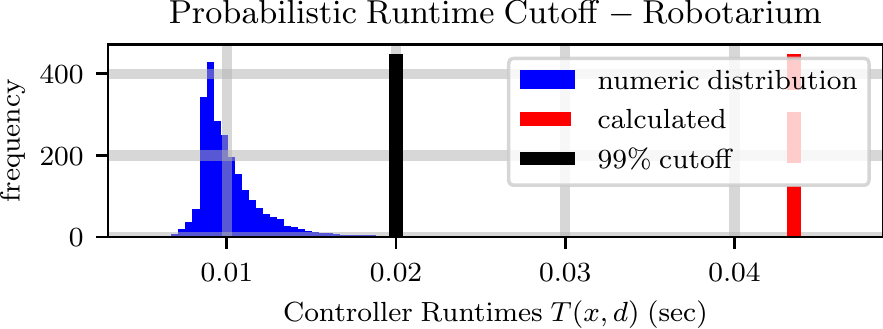}
    \caption{Robotarium Hardware data when (top) taking a percentile method to solving~\eqref{eq:augmented_NMPC}, and (bottom) calculating a probabilistic cutoff on maximum controller runtime.  In both cases, the red lines corresponding to (top) the identified waypoint and (bottom) the reported maximum controller runtime are to the left and right, respectively, of their corresponding, true probabilistic cutoffs.  In other words, the identified values satisfy their corresponding probabilistic statements, affirming Corollaries~\ref{corr:percent_opt_FTOCP} and~\ref{corr:run_times}.  Numeric distributions were calculated by evaluating $5000$ random samples.} 
    \vspace{-0.2 in}
    \label{fig:robotarium_data}
\end{figure}

\newidea{NMPC Formulation:} Based on the setup of static obstacles and goal locations on the grid, we define a function $S: \mathcal{W} \to \mathbb{R}_{+}$ that outputs the length of the shortest feasible path to a goal from a provided planar waypoint.  Should no feasible path exist from a waypoint $w \in \mathcal{W}$, $S(w) = 100$ to indicate infeasibility.  Inspired by discrete control barrier function theory~\cite{agrawal2017discrete}, we define a control barrier function $h$ which accounts for both the ego agent state $x_a$ and the un-controlled agent state $x_o$ (with $P = [I_{2 \times 2}~\mathbf{0}_{2 \times 1}]$):
\begin{equation}
    h(x_a,x_o) = \begin{cases}
        -5 & \mbox{in~static~obstacle~cell}, \\
        \|P(x_a - x_o)\| - 0.18 & \mbox{else}.
    \end{cases}
\end{equation}
Then, provided $h(x_a,x_o) \geq 0$, the ego agent hasn't crashed into a static obstacle and is maintaining at least a distance of $0.18$ m from the un-controlled agent.

This permits us to define an NMPC problem as follows with the dynamics $f$ as per~\eqref{eq:unicyle_model} and $\forall~j \in [1,2,3,4,5]$:
\begin{align}
w^*_k = &~\argmin_{w \in \mathcal{W}}~& & S(w), \label{eq:rob_NMPC} \tag{NMPC-A} \\
& \mathrm{subject~to~} & & x^{j}_{k} = f(x^{j-1}_k,u^{j-1},d), \label{eq:constr_1} \tag{a}\\
& & & x^0_k = x_k, \label{eq:constr_2} \tag{b} {\color{white} \eqref{eq:constr_2},\eqref{eq:constr_3}}\\
& & & h(x^j_{k,a},x_o) \geq 0 \label{eq:constr_3} \tag{c} \\
& & & u^{j-1} = U\left(x^{j-1}_k, d, w\right), \label{eq:constr_4} \tag{d} \\
& & & 0.05 \leq \|w - x_k\| \leq 0.2.
\end{align}
To ease sampling then, we will consider an augmented cost $J$ that outputs $100$ whenever a waypoint $w$ fails to satisfy constraints~\eqref{eq:constr_1}-\eqref{eq:constr_4} in~\eqref{eq:rob_NMPC}.  Then we define the NMPC problem to-be-solved as follows:
\begin{align}
w^*_k = &~\argmin_{w \in \mathcal{W}}~& & J(w), \label{eq:augmented_NMPC} \tag{NMPC-B} \\
& \mathrm{subject~to~} & & 0.05 \leq \|w - x_k\| \leq 0.2.
\end{align}

\begin{figure}[t]
    \centering
    \includegraphics[width = \columnwidth]{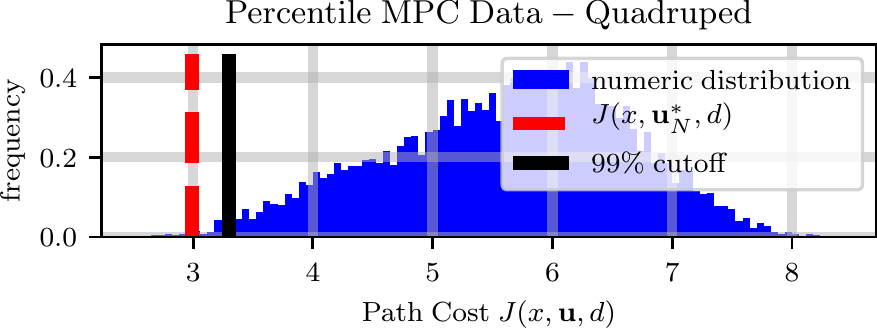} \vspace{0.025 in}\\
    \includegraphics[width = \columnwidth]{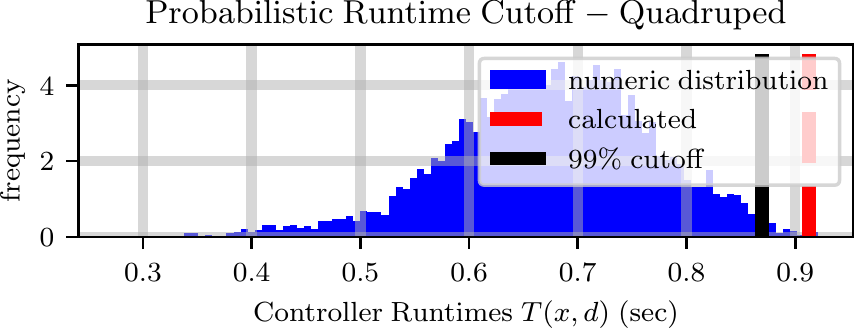}
    \caption{Quadruped Hardware data when (top) taking a percentile method to solve~\eqref{eq:quad_FTOCP}, and (bottom) calculating a probabilistic cutoff on maximum controller runtime.  In both cases, the red lines corresponding to (top) the identified path and (bottom) the reported maximum controller runtime are to the left and right, respectively, of their corresponding, true probabilistic cutoffs.  This affirms Corollaries~\ref{corr:percent_opt_FTOCP} and~\ref{corr:run_times} insofar as the identified values satisfy their corresponding probabilistic statements.  Numeric distributions were calculated by evaluating $5000$ random samples.} 
    \vspace{-0.2 in}
    \centering
    \label{fig:quadruped_data}
\end{figure}

\newidea{Results:} By Corollary~\ref{corr:percent_opt_FTOCP}, if we wish to take a percentile approach to determine a waypoint $w^*_N$ in the $95\%$-ile with $99.4\%$ confidence we need to evaluate $N = 100$ uniformly chosen waypoints from the constraint space for~\eqref{eq:augmented_NMPC}.  Figure~\ref{fig:robotarium_data} shows the cost of the outputted waypoint sequence compared against $5000$ randomly sampled values, and as can be seen, the outputted waypoint $w^*_N$ is indeed in the $95\%$-ile, confirming Corollary~\ref{corr:percent_opt_FTOCP}.  Calculating this controller's runtime in $460$ randomly sampled initial state and environment scenarios yielded a probabilistic maximum $\zeta^*_N = 0.043$ seconds.  According to Corollary~\ref{corr:run_times}, this maximum runtime should be an upper bound on the true, $99\%$ cutoff on controller runtimes with confidence $99\%$ --- and as can be seen in Figure~\ref{fig:robotarium_data}, $\zeta^*_N$ exceeds the true value.  Finally, to corroborate Corollary~\ref{corr:recursive_feasible}, we evaluated the recursive feasibility cost function $C$ as per~\eqref{eq:recursive_feasibility_cost} in each of the same $460$ randomly sampled scenarios from prior.  In each scenario, the percentile controller was successively feasible, indicating that with $99\%$ probability the controller will be successively feasible.  Evaluating the same cost for $5000$ more uniformly chosen samples resulted in the controller being successively feasible each time, corroborating Corollary~\ref{corr:recursive_feasible}.

\begin{figure}[t]
    \centering
    \includegraphics[width = \columnwidth]{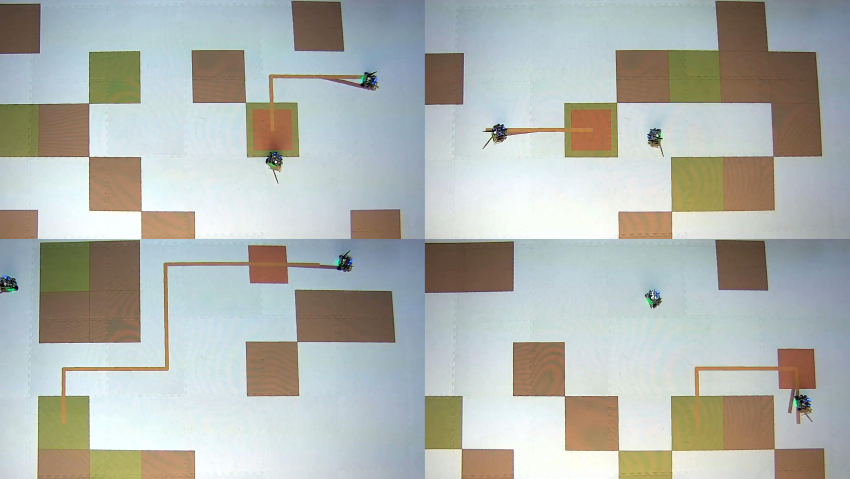}
    \caption{Experimental depictions of the randomized environments $\mathcal{D}$ for the Robotarium as described in Section~\ref{sec:robotarium}. The black squares correspond to static obstacles, the green squares correspond to goals for the ego-agent whose shortest path from its starting cell is shown in orange, and the red squares correspond to the un-controlled agent's goal.}
    \vspace{-0.2 in}
    \label{fig:example_sim_setup}
\end{figure}

\section{Conclusion}
Based on existing work in black-box risk-aware verification, we provided probabilistic guarantees for percentile approaches to solving finite-time optimal control problems, recursive feasibility of such approaches, and bounds on maximum controller runtimes.  In future work, the authors plan to explore how the generated probabilistic guarantees can be applied in other scenarios, \textit{e.g.} probabilistic planning procedures.  Secondly, we aim to bound the optimality gap between our percentile solutions and the global optimum.

\bibliographystyle{IEEEtran}
\bibliography{IEEEabrv,bib_works}

\begin{figure}[t]
    \centering
    \includegraphics[width = \columnwidth]{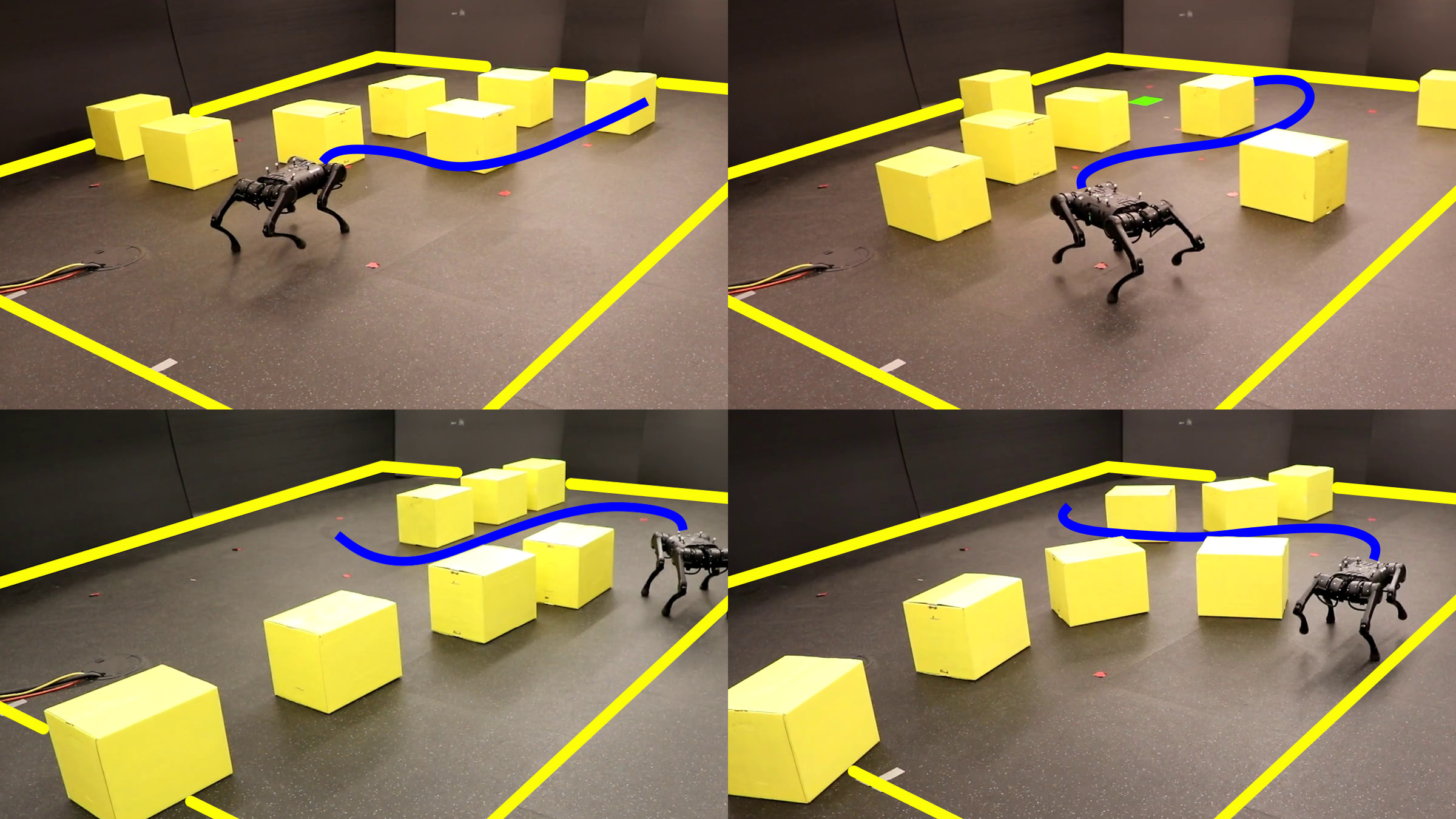}
    \caption{Depictions of the randomized environments $\mathcal{D}$ for the Quadruped experiments as described in Section~\ref{sec:quadruped}. Yellow boxes are static obstacles, and the goal is shown in green (not visible in all images). The computed plan is depicted in blue.}
    \vspace{-0.2 in}
    \label{fig:example_quad_setup}
\end{figure}

\end{document}